\documentclass{article}
\usepackage[utf8]{inputenc}
\usepackage[english]{babel}
\usepackage[top=1in,bottom=1in,left=1in,right=1in]{geometry}
 
\usepackage{graphicx}

\usepackage{mathtools} 
\usepackage{amssymb}
\usepackage{float}
\usepackage{subfigure}
\usepackage{tikz}

\usepackage{multirow}
\usepackage{booktabs}

\usepackage{amsmath, amsfonts}
\usepackage{amsthm}
\newtheorem{theorem}{Theorem}[section]
\newtheorem{lemma}{Lemma}[section]
\newtheorem{remark}{Remark}[section]

\usepackage{algorithm,algpseudocode}

\usepackage{hyperref}

\newcommand{\R}{\mathbb R}
\newcommand{\E}{\mathbb E}
\newcommand{\N}{\mathbb N}
\newcommand{\Hi}{\mathcal H}

\newcommand{\argmin}{\text{argmin}}
\newcommand{\opt}{^{\star}}

\newcommand{\ALG}{\texttt{ALG}}
\newcommand{\OLNM}{\texttt{OLNM}}
\newcommand{\ORGD}{\texttt{ORGD}}

\renewcommand{\SS}{\mathcal{S}} 

\newcommand{\rev}[1]{{\color{black} #1}}

\begin{document}

\title{Bounds for the tracking error of  first-order online optimization methods}

\author{Liam Madden\thanks{\url{liam.madden@colorado.edu}, Dept.\ of Applied Mathematics, University of Colorado Boulder}  \and Stephen Becker\thanks{\url{stephen.becker@colorado.edu}, Dept.\ of Applied Mathematics, University of Colorado Boulder}   \and  Emiliano Dall'Anese\thanks{\url{emiliano.dallanese@colorado.edu}, Dept.\ of Electrical, Computer and Energy Engineering, University of Colorado Boulder} }

\date{\today}

\maketitle

\begin{abstract}

This paper investigates online algorithms for smooth time-varying optimization problems, focusing first on methods with constant step-size, momentum, and extrapolation-length. Assuming strong convexity, precise results for the tracking iterate error (the limit supremum of the norm of the difference between the optimal solution and the iterates) for online gradient descent are derived. The paper then considers a general first-order framework, where a universal lower bound on the tracking iterate error is established. Furthermore, a method using ``long-steps'' is proposed and shown to achieve the lower bound up to a fixed constant. This method is then compared with online gradient descent for specific examples. Finally, the paper analyzes the effect of regularization when the cost is not strongly convex. With regularization, it is possible to achieve a non-regret bound. The paper ends by testing the accelerated and regularized methods on synthetic time-varying least-squares and logistic regression problems, respectively.

\end{abstract}

\section{Introduction}
Modern optimization applications have increased in scale and complexity~\cite{SPMag}, and furthermore, some applications require solutions to a series of problems with low latency. For example, in contrast to legacy power distribution systems that were built for constant and unidirectional flow, modern power systems that include solar power at the residential nodes must incorporate variable and bidirectional flow. Thus, in order to preserve efficiency, control decisions, solved via optimization, need to be made frequently, at the time scale of changing renewable generation and non-controllable loads (e.g., seconds). Making the problem even harder is that the problem must be solved for a greater number of control points, and so it takes longer to find a suitable solution. In this example, and many others, batch algorithms take longer to find a suitable solution than is allowed, and so we have to abandon them in favor of online algorithms. Motivated by applications requiring a time-varying framework, such as power systems \cite{opf,Tang}, transportation systems \cite{Bianchin}, and communication networks~\cite{Chen12}, this paper evaluates such online algorithms.

In particular, we consider online first-order algorithms. If it takes time $h>0$ to make one gradient call (or one gradient call and one evaluation of the proximal operator), then we encode this into the time-varying problem by discretizing the cost function with respect to $h$, leading to a sequence of cost functions. The goal of an online algorithm is to track the minimum of this sequence, taking only one step at each time.

In the presence of strong convexity, upper bounds for online gradient descent have been proved in, e.g.,  \cite{Emiliano,Andrea}. \rev{For completeness, we include such an upper bound in Theorem \ref{thm:stronglyConvex1}. Furthermore,} for the first time it is established (Theorem~\ref{thm:stronglyConvex2}) that this is a tight bound. Beyond online gradient descent, there are dynamic regret bounds for online accelerated methods, but these are not shown to be optimal \cite{Yang}. This paper goes further. First, Theorem \ref{thm:firstorder} proves a lower bound for online first-order methods in general. Then, we define a method that we call online long-step Nesterov's method, and prove a proportional upper bound for it in Theorem~\ref{thm:onlineNesterov}. In the absence of strong convexity, we show, in Theorem~\ref{thm:onlineRegularization}, that regularizing online gradient descent leads to a useful bound that is not in terms of the regret. Finally, this paper demonstrates the performance of the algorithms by applying them to synthetic time-varying least squares and logistic regression problems.

The rest of the paper is organized as follows. Section \ref{sec:preliminaries} provides all the necessary preliminaries. Section \ref{sec:gradient} considers online gradient descent, online Polyak's method, and online Nesterov's method. Section \ref{sec:Nesterov} considers the full class of online first-order methods, including online long-step Nesterov's method. Section \ref{sec:regularization} details the results for online regularized gradient descent. Section \ref{sec:numerics} demonstrates the performance of the algorithms on some numerical examples. Section \ref{sec:conclusion} concludes the paper.

\section{Preliminaries} \label{sec:preliminaries}

This section reviews useful results from convex analysis~\cite{Beck,Combettes,Nesterov,Nesterov2,Boyd,Bubeck}, and introduces key definitions that will be used throughout the paper. The paper considers functions on a Hilbert space $\Hi$ with corresponding norm $\|\cdot\|$. We assume all functions, $f_t: \Hi \rightarrow \R$,  are  proper, lower semicontinuous, convex, and strongly smooth; $t\in\N\cup\{0\}$ denotes the time index~\cite{SPMag,popkov2005gradient,SimonettoGlobalsip2014,mokhtari2016online}. The discretized time-varying optimization problem of interest is the sequence of convex problems (one for each $t$): 
\begin{align}
\label{eq:time-varying-opt}
   \min_{x\in\Hi} f_t(x) \,, \hspace{.5cm} t \in\N\cup\{0\}
\end{align}
along with a time-varying minimizer set. In particular, assume that the minimizer set, denoted as $X_t\opt$, is nonempty for each $t$. Let $x_t\opt$ denote an element of $X_t\opt$ and $f_t\opt=f_t(x_t\opt)$. 
We denote the set of sequences of $L$-strongly smooth functions by $\SS'(L)$. For function sequences that are additionally $\mu$-strongly convex, $X_t\opt$ contains only one point; therefore, we can measure the temporal variability of the optimal solution trajectory of~\eqref{eq:time-varying-opt} with the sequence 
\begin{align*}
    \sigma_t :=\|x_{t+1}\opt-x_t\opt\| , \hspace{.5cm} t\in \N\cup\{0\}.
\end{align*}
We define $\SS(\kappa^{-1},L,\sigma)$, where $\kappa$ is the condition number, as the set of $L$-strongly smooth, $\kappa^{-1}L$-strongly convex functions sequences with $\sigma_t \leq \sigma$ for all $t$. For both $\SS'(L)$ and $\SS(\kappa^{-1},L,\sigma)$, the Hilbert space, along with its dimension, is left implicit since the results are independent of it.

Throughout the paper, we will consider various measures of optimality~\cite{popkov2005gradient,SimonettoGlobalsip2014,mokhtari2016online,SPMag,hall2015online,Jadbabaie2015,besbes2015non,yi2016tracking,Emiliano}: the iterate error $\|x_t-x_t\opt\|$, the function error $f_t(x_t)-f_t\opt$, and the gradient error $\|\nabla f_t(x_t)\|$. For functions that are not strongly convex, the iterate error yields the  strongest results, while the gradient error is the weakest. That is,
\begin{align}
 \|\nabla f_t(x) \| &\leq L \|x-x_t\opt\|,\label{eq:gradientBound} \\
 f_t(x) - f_t\opt &\leq \frac{L}{2}\|x-x_t\opt\|^2,\\
 \text{and }\|\nabla f_t(x) \|^2 &\leq 2L \left( f_t(x) - f_t\opt\right)
\end{align}
where the first two inequalities follow from standard arguments in convex analysis, \rev{and the third inequality follows by comparing a point and a gradient step from it in the definition of strong smoothness \cite[Sec. 12.1.3]{Shai}}. For functions that are strongly  convex, bounds in the opposite directions can be found.

Due to the temporal variability, we cannot expect any of the error sequences to converge to zero in general. Thus, we will characterize the performance of online algorithms via bounds on the limit supremum of the errors, which we term ``tracking'' error, rather than bounds on the convergence rate (to the tracking error ball).

\rev{Note that for $\SS'(L)$, the regret literature uses the dynamic regret, $Reg_T\coloneqq \sum_{t=0}^T f_t(x_t)-f_t\opt$, as a measure of optimality~\cite{besbes2015non,hall2015online,yi2016tracking,Jadbabaie2015}. The path variation, function variation, and gradient variation---respectively,
\begin{align*}
    V_T^p &= \max_{\{x_t\opt\in X_t\opt\}_{t=0}^T}\sum_{t=1}^T \|x_t\opt-x_{t-1}\opt\|, \\
    V_T^f &= \sum_{t=1}^T\sup_x |f_t(x)-f_{t-1}(x)|, \\
\text{and}\;
  V_T^g &= \sum_{t=1}^T \sup_x \|\nabla f_t(x)-\nabla f_{t-1}(x)\|
\end{align*}
---are used to bound the dynamic regret. On the other hand, our approach to $\SS'(L)$ is inspired by regularization reduction \cite{ZhuReduction}. Through regularization, we are able to bound the tracking gradient error via the $(\sigma_t)$ corresponding to the regularized problem. But, there are a couple of reasons why our bound cannot be compared with the bounds in the regret literature. First, it is possible for $Reg_T/(T+1)$ to be bounded while the function error has a subsequence going to infinity. For example, let $(\Delta_t)_{t\in\N\cup\{0\}}=(1,0,2,0,0,3,0,0,0,\ldots)$. Then $\frac{1}{T+1}\sum_{t=0}^T \Delta_t\leq 1$ even though there is a subsequence of $(\Delta_t)$ that goes to infinity. Second, in order for the function variation or gradient variation to be finite, the constraint set must be compact. Our bound only requires that the optimal solution trajectory lie in a compact set.}

In this paper, we consider three general algorithms: $\ALG(\alpha,\beta,\eta)$ presented in Section~\ref{sec:gradient}, online long-step Nesterov's method ($\OLNM(T)$) presented in Section~\ref{sec:Nesterov}, and online regularized gradient descent ($\ORGD(\delta,x_c)$), presented in Section~\ref{sec:regularization}. $\ALG(\alpha,\beta,\eta)$ encompasses methods such as online gradient descent, online Polyak's method (also known as the online heavy ball method), and online Nesterov's method as special cases (see Table~\ref{table:ALG}).

The proposed $\OLNM$ is motivated by the following observation: suppose that the optimizer has access to all the previous functions; depending on the temporal variability of the problem,  the question we pose is whether  it is better to use the same function for some number of iterations before switching to a new one, or utilize a new function at each iteration. This question is relevant especially in the case where ``observing'' a new function may incur a cost (for example, in sensor networks, where acquisition of new data may be costly from a battery and data transmission standpoint).  We call the former idea ``long-steps,'' drawing an analogy to interior-point methods \cite[Sec. 14]{Nocedal},~\cite[Sec. 11]{Boyd}. Specifically, we pause at a function for some number of iterations, and we apply a first-order algorithm with the optimal coefficients for the number of iterations, in the sense of \cite{Drori,Kim,Taylor}. Thus, short-steps are online gradient steps. Restarting Nesterov's method has been shown to be optimal for strongly convex functions in, e.g. \cite{Candes},~\cite[App. D]{ZhuUnification}. It turns out that the optimal restart count as a long-step length is optimal for the strongly convex time-varying problem.

On the other hand, the literature on batch algorithms with inexact gradient calls has shown that accelerated methods accumulate errors for some non-strongly convex functions \cite[Prop. 2]{Schmidt},~\cite[Thm. 7]{Devolder},~\cite{Bertsekas,Villa,Aujol}. This suggests that there may be some non-strongly convex function sequences such that all online accelerated methods perform worse than online gradient descent. On the other hand, only averaged function error bounds exist for online gradient descent. We get around this via regularization; in particular, we balance the error introduced by the regularization term with a smaller tracking error achieved by the regularized algorithm. This allows us to derive a non-averaged error bound for $\ORGD$.

While the paper considers strongly smooth functions for simplicity of exposition, we note how to extend results to the composite case (where the cost function is the sum of $f$ and a possibly nondifferentiable function with computable proximal operator, such as an indicator function encoding constraints or an $\ell_1$ penalty). We leave the analysis for Banach spaces as a future direction. See \cite[Sec. 3.1]{Tseng} for an excellent treatment of acceleration for Banach spaces.

\section{Simple first-order methods} \label{sec:gradient}

In this section, we focus on the class of algorithms that, given $x_0$, construct $(x_t)$ via:
\begin{align}
    x_{t+1} &= x_t-\alpha\nabla f_t(y_t)+\beta(x_t-x_{t-1}) \tag{$\ALG(\alpha,\beta,\eta)$}\\
    y_{t+1} &= x_{t+1}+\eta(x_{t+1}-x_{t})\notag
\end{align}
which will be referred to as $\ALG(\alpha,\beta,\eta)$, where $\alpha > 0$ is the step-size, $\beta$ is the momentum, and $\eta$ is the extrapolation-length. $\ALG(\alpha,0,0)$ corresponds to online gradient descent, $\ALG(\alpha,\beta,0)$ to an online version of Polyak's method \cite{Polyak}, and $\ALG(\alpha,\beta,\beta)$ to an online version of Nesterov's method \cite{Nesterov2}. 

\begin{table}
\begin{tabular}{lll}
\toprule
Form of $\ALG(\alpha,\beta,\eta)$ & Typical parameter choice & Name \\
\midrule
$\ALG(\alpha,0,0)$ & $0 < \alpha<\frac{2}{L}$ & Online gradient descent  \\
 $\ALG(\alpha,\beta,0)$ & $\alpha=\frac{4}{\left(\sqrt{L}+\sqrt{\mu}\right)^2}$, $\beta=\left(\frac{\sqrt{L}-\sqrt{\mu}}{\sqrt{L}+\sqrt{\mu}}\right)^2$ & Online Polyak's method \cite{Polyak}  \\
  $\ALG(\alpha,\beta,\beta)$ & $\alpha=1/L$,  $\beta = \eta= \frac{\sqrt{L}-\sqrt{\mu}}{\sqrt{L}+\sqrt{\mu}}$ & Online Nesterov's method \cite{Nesterov2} \\
\bottomrule     
\end{tabular}
\caption{Special cases of $\ALG(\alpha,\beta,\eta)$
for $(f_t)\in \SS(\kappa^{-1},L,\sigma)$ where $\mu = L/\kappa$.
}
\label{table:ALG}
\end{table}

It should be pointed out that, while the analysis of online algorithms for time-varying optimization such as $\ALG(\alpha,\beta,\eta)$ share commonalities with online learning~\cite{Hazan2007,hall2015online,Jadbabaie2015}, the two differ in their motivations and aspects of their implementations, as noted in~\cite{SPMag}. For example, in online learning frameworks, the step size may depend on the time-horizon or, in the case of an infinite time-horizon, a doubling-trick may be utilized \rev{\cite[Sec. 2.3.1]{Shalev}}. On the other hand, every step of an online algorithm applied to a time-varying problem is essentially the \textit{first} step at that time. Hence, the parameters of the algorithm should be cyclical or depend on measurements of the temporal variation. We only consider the former. The simplest subset of cyclical algorithms is the set of algorithms with constant parameters, as in $\ALG$.

In the following subsections, we give a tight bound on the performance of online gradient descent and analyze $\ALG$ for two examples.

\subsection{Online gradient descent}

When the function $f_t$ is strongly convex for all $t$, upper bounds on the tracking errors for online gradient descent are available in  the literature (see, e.g.,~\cite{Emiliano,Andrea,Liam,dixit2019online}); these results are tailored to the setting considered in this paper by the following theorem, which is followed by two examples used to give intuition and prove tightness results.

\begin{theorem} \label{thm:stronglyConvex1}
Suppose that $(f_t)\in \SS(\kappa^{-1},L,\sigma)$ and let \rev{$\alpha\in ]0,2/(\mu+L)]$}. Then, given $x_0$, $\ALG(\alpha,0,0)$  constructs a sequence $(x_t)$ such that
\begin{align}
    \label{eq:linear_OGD}
    \limsup_{t\to\infty}\|x_t-x_t\opt\| \leq (\alpha\mu)^{-1}\sigma
\end{align}
where $\mu=\kappa^{-1}L$. In particular, the bound is minimized for $\alpha=\frac{2}{\mu+L}$, in which case,
\begin{align}
    \limsup_{t\to\infty}\|x_t-x_t\opt\| \leq \frac{\kappa+1}{2}\sigma.
\end{align}
\end{theorem}
The proof follows by using the fact that the map $I-\alpha\nabla f_t$ is Lipschitz continuous with parameter $c = \max\{|1-\alpha\mu|,|1-\alpha L|\}$ \cite[Sec. 5.1]{BoydPrimer} and $x_t\opt$ is a fixed point of that map. The result straightforwardly extends to the composite case by using the nonexpansiveness property of the proximal operator and a similar fixed point result.

\subsection{Example: translating quadratic}

For quadratic functions with constant positive definite matrix but minimizer moving at constant speed on a straight line, the analysis of \cite{Lessard} can be extended, using the Neumann series, to show that the set of $(\alpha,\beta,\eta)$ such that $\ALG$ has a finite worst-case tracking iterate error is exactly the stability set of $(\alpha,\beta,\eta)$ in the batch setting (i.e., in a setting where the cost does not change during the execution of the algorithm). As an example, the stability set for online Nesterov's method is given in \cite[Prop. 3.6]{Aybat}. Thus, in the online setting, one should consider  only $(\alpha,\beta,\eta)$ that are in the batch stability set. 

\newcommand{\DD}{\xi}
\rev{As a particular example, let $f(x)=\frac{1}{2}x^TAx$, with $A=\text{diag}(\mu,L,L,...)$. Given $(\alpha,\beta,\eta)$ and initialization $x_0$, we want to construct $(f_t)\in\SS(\mu/L,L,\sigma)$ in such a way that the iterates of $\ALG(\alpha,\beta,\eta)$ trail behind $(x_t\opt)$ at a constant distance. Towards this end, define $\DD=\left(\frac{1-\beta}{\alpha\mu}+\eta\right)\sigma$, which will end up being the trailing distance. Let $e$ denote the first canonical basis vector, and define $f_t=f(\cdot-x_t\opt)$ where $x_t\opt=x_0+(t\sigma +\DD)e$. By induction, we will show that $\Delta_t \coloneqq x_t-x_t\opt=-\DD e~\forall t\in\N\cup\{0\}$.} The base case follows by construction. 
Now, assume the result holds for all 
\newcommand{\kk}{t}
$t'\leq \kk$. Then,
\begin{align*}
    \Delta_{\kk+1} &= (1+\beta)x_{\kk}-\beta x_{\kk-1}-x_{\kk+1}\opt-\alpha\nabla f_{\kk}\big((1+\eta)x_{\kk}-\eta x_{\kk-1}\big)\\
    &= (1+\beta)\Delta_{\kk}-\beta\Delta_{\kk-1}+(1+\beta)(x_{\kk}\opt-x_{\kk+1}\opt)-\beta(x_{\kk-1}\opt-x_{\kk+1}\opt)\\
    &\hspace{1cm}-\alpha\nabla f\left((1+\eta)\Delta_{\kk}-\eta \Delta_{\kk-1}-\eta(x_{\kk-1}\opt-x_{\kk}\opt)\right)\\
    &= -(1+\beta)\DD e+\beta \DD e-(1+\beta)\sigma e+2\beta\sigma e\\
    &\hspace{1cm}-\alpha\nabla f\left(-(1+\eta)\DD e+\eta \DD e+\eta \sigma e\right)\\
    &= -\DD e - (1-\beta)\sigma e - \alpha \nabla f(-\DD e+\eta\sigma e)\\
    &= -\DD e +\alpha\mu \DD e-(1-\beta+\alpha\mu\eta)\sigma e\\
    &= -\DD e
\end{align*}
and so the result holds for all $t$ by induction. Figure \ref{fig:diagram3} shows how the iterates trail behind the minimizers.


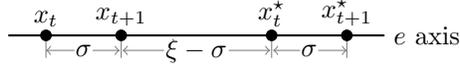
\begin{figure}[H]
\centering
\begin{tikzpicture}
\centering
\draw[black, thick] (-2.5,0) -- (2.5,0);
\foreach \x in {-2,-1,1,2}
    \draw[gray] (\x,0) -- (\x,-.3);
\filldraw[black] (-2,0) circle (2pt) node[anchor=south] {$x_t$};
\draw[<-,gray] (-2,-.2) -- (-1.62,-.2);
\draw[->,gray] (-1.38,-.2) -- (-1.02,-.2);
\draw[<-,gray] (2,-.2) -- (1.62,-.2);
\draw[->,gray] (1.38,-.2) -- (1.02,-.2);
\draw[<-,gray] (-.98,-.2) -- (-.43,-.2);
\draw[<-,gray] (.98,-.2) -- (.43,-.2);
\filldraw[black] (-1,0) circle (2pt) node[anchor=south] {$x_{t+1}$};
\filldraw[black] (1,0) circle (2pt) node[anchor=south] {$x\opt_t$};
\filldraw[black] (2,0) circle (2pt) node[anchor=south] {$x\opt_{t+1}$};
\node[] at (-1.5,-.2) {$\sigma$};
\node[] at (0,-.2) {$\DD-\sigma$};
\node[] at (1.5,-.2) {$\sigma$};
\node[anchor=west] at (2.5,0) {\rev{$e$} axis};
\end{tikzpicture}
\caption{Movement of iterates and minimizers}
\label{fig:diagram3}
\end{figure}

\noindent
For online gradient descent, we have $\DD=(\alpha\mu)^{-1}\sigma$; this  shows that the bound in Theorem \ref{thm:stronglyConvex1} is tight. We formalize this  tightness result (which is a contribution of the present paper) in Theorem \ref{thm:stronglyConvex2}.

\begin{theorem} \label{thm:stronglyConvex2}
For all \rev{$\alpha\in]0,2/(\mu+L)]$} and initialization $x_0$, $\exists(f_t)\in\\ \SS(\mu/L,L,\sigma)$ such that $\ALG(\alpha,0,0)$  constructs a sequence $(x_t)$ with
\begin{align}
    \label{eq:lowerbound_OGD}
    \limsup_{t\to\infty}\|x_t-x_t\opt\| = (\alpha\mu)^{-1}\sigma.
\end{align}
\end{theorem}

For Polyak's method, using the usual parameters, one gets $\rev{\DD}=\sqrt{\kappa}\sigma$. On the other hand, for Nesterov's method one gets $\rev{\DD}=\left(2\sqrt{\kappa}-1\right)\sigma$. Thus, when applied to this example, the tracking iterate error scales with the square root of the condition number for both of these methods. However, as we will see in the next example, this is not always the case for these methods.

\subsection{Example: rotating quadratic}
Consider the function  $f$ utilized in the previous example, and consider rotating the first two canonical basis directions every iteration. We can reduce the full problem to one in $\R^2$. Define
\begin{align*}
f_t(x)&=\frac{1}{2}x^T A_t x\\
A_{2t}&=\begin{pmatrix} L&0\\0&\mu \end{pmatrix}\\
A_{2t+1}&=\begin{pmatrix} \mu&0\\0&L \end{pmatrix}
\end{align*}
and note that $(f_t)\in \SS(\mu/L,L,\sigma)$ for all $\sigma\geq 0$. Now, let $(x_t)$ be the sequence generated by $\ALG(\alpha,\beta,\eta)$ given $x_0$. Denote $a_+=(1+\beta)-(1+\eta)\alpha L$, $a_-=(1+\beta)-(1+\eta)\alpha \mu$,  $b_+=-\beta+\eta\alpha L$, and $b_-=-\beta+\eta\alpha \mu$. Then,
\begin{align*}
x_{2t} &= \begin{pmatrix}a_-&0\\0&a_+ \end{pmatrix}x_{2t-1}+ \begin{pmatrix}b_-&0\\0&b_+\end{pmatrix}x_{2(t-1)}\\
x_{2t+1} &= \begin{pmatrix}a_+&0\\0&a_- \end{pmatrix}x_{2t}+ \begin{pmatrix}b_+&0\\0&b_-\end{pmatrix}x_{2t-1}\\
&= \begin{pmatrix}a_+a_-+b_+&0\\0&a_+a_-+b_- \end{pmatrix}x_{2t-1}+ \begin{pmatrix}a_+b_-&0\\0&a_-b_+\end{pmatrix}x_{2(t-1)}.
\end{align*}
Define $z_t=[x_{2t};~x_{2t+1}]$. Then,
\begin{align*}
z_t &= \begin{pmatrix}
b_-&0&a_-&0\\0&b_+&0&a_+\\a_+b_-&0&a_+a_-+b_+&0\\0&a_-b_+&0&a_+a_-+b_-
\end{pmatrix}
z_{t-1}\\
&\coloneqq C z_{t-1}\\
&= C^t z_0.
\end{align*}
Thus, $z_t\to z\opt=0$ precisely when $\rho(C)<1$ (since $\rho(C)=\lim_t \|C^t\|^{1/t}$). It is easy to see that $C$ is similar to the block-diagonal matrix with $D$ and $E$ on the diagonal blocks where
\begin{align*}
D=\begin{pmatrix}
b_-&a_- \\  a_+b_- & \; a_+a_-+b_+
\end{pmatrix}
\text{ and }
E=\begin{pmatrix}
b_+&a_+ \\  a_-b_+& \; a_+a_-+b_-.
\end{pmatrix}
\end{align*}
Furthermore, $D$ and $E$ have the same trace and determinant:
\begin{align*}
\text{tr}(D)=\text{tr}(E)&= a_+a_-+b_++b_-\\
\text{det}(D)=\text{det}(E)
&= a_+a_-b_-+b_+b_--a_+a_-b_- 
= b_+b_-.
\end{align*}
Thus, $\rho(C)=\rho(D)=\rho(E)$. Note that $\rho(C)=0$ when $\alpha=\frac{1}{L}$ and $\beta=\eta=0$. In fact, $\ALG(1/L,0,0)$ converges exactly in just two steps. Thus, online gradient descent performs better than online Polyak's method and online Nesterov's method for the rotating quadratic example. In fact, online Polyak's method actually \textit{diverges}. For the usual parameters of Polyak's method, $\rho(C)=6\left(\frac{\sqrt{\kappa}-1}{\sqrt{\kappa}+1}\right)^2$, which is bigger than 1 precisely when $\kappa>\left(\frac{\sqrt{6}+1}{\sqrt{6}-1}\right)^2\approx 5.7$. We state this more formally in Theorem \ref{thm:polyak}.

\begin{theorem}
\label{thm:polyak}
For all $\kappa\geq 6$, $L>0$, $\sigma\geq 0$, $\exists(f_t)\in \SS(\kappa^{-1},L,\sigma)$, such that online Polyak's method diverges (for any non-optimal initialization $x_0$).
\end{theorem}

In the batch setting, decreasing the step-size increases robustness to noise \cite{Lessard,Cyrus}. Thus, it makes sense that in the online setting, decreasing the step-size leads to greater stability. In other words, online Polyak's method diverges because of its large step-size. 
By reducing the stepsize to 
$\alpha=\frac{1}{L}$, then $\eta\leq\beta<1$ is sufficient to guarantee $\rho(C)<1$.

\section{General first-order methods} \label{sec:Nesterov}

\rev{In \cite{Nemirovsky}, Nemirovsky and Yudin proved lower bounds on the number of first-order oracle calls necessary for $\epsilon$-convergence in both the smooth and convex setting and the smooth and strongly convex setting \cite[Thm. 7.2.6]{Nemirovsky}. Then, in \cite{nesterov1983}, Nesterov constructed a method that achieved the lower bound in the smooth and convex setting. His method has a momentum parameter that goes to one as the iteration count increases. By modifying the momentum sequence, it is possible to achieve the lower bound in the smooth and strongly convex setting as well. This can be done by either setting the momentum parameter to a particular constant, restarting the original method every time the iteration count reaches a particular number, or adaptively restarting the original method \cite{Candes}. 

Nesterov presents the lower bounds in Theorems 2.1.7 and 2.1.13 respectively of \cite{Nesterov2}. 
These theorems involve some subtleties, which we now discuss.
First, Nesterov says that $(x_t)$ comes from a first-order method if $x_{t+1}-x_0\in \text{span}\{\nabla f(x_0),...,\nabla f(x_t)\}$ for all $t$. This is the definition that we will generalize to the time-varying setting. Second, the lower bounds do not hold for all $t$. The lower bound in the smooth and convex setting only holds for $t<\frac{1}{2}(d-1)$ where $d$ is the dimension of the space. In the smooth and strongly convex setting, Nesterov only 
proves the bound
for infinite-dimensional spaces. In fact, for finite-dimensional spaces, the lower bound can only hold for $t\leq O(d)$ since the conjugate gradient method applied to quadratic functions converges exactly in $d$ iterations. In order to prove lower bounds that hold for all $t$ in finite-dimensional spaces, it is necessary to restrict to smaller classes of methods.
In particular, \cite{arjevani2016lower} excludes the conjugate gradient method by restricting to methods with ``stationary'' update rules. Third, the lower bounds are based on explicit adversarial functions. In particular, we will use the adversarial function that Nesterov gives in \cite[2.1.13]{Nesterov2} to construct an adversarial sequence of functions in the time-varying setting. Thus, our lower bound only holds for infinite-dimensional spaces. It is an open problem whether the function sequence can be modified to give a lower bound that holds for all $t$ in finite-dimensional spaces or whether it is necessary to restrict to smaller classes of methods.

We say that $(x_t)$ comes from a first-order method if $x_{t+1}-x_0\in\\ \text{span}\{\nabla f_{\tau_0}(x_0),...,\nabla f_{\tau_t}(x_t)\}$ where, for each $t$, $\tau_t\in \{0,...,t\}$.} More generally, we still consider $(x_t)$ to be from a first-order method if $(x_t)$ is a simple auxiliary sequence of some $(y_t)$ that is more precisely first-order. Now, calling the most recent gradient at each step of the algorithm, as $\ALG$ does, corresponds to $\tau_t = t$ for all $t$. While it is possible for an algorithm to call an older gradient, it is not clear how this could be helpful. In section \ref{subsec:OLNM} will show that having $\tau_t = T\lfloor t/T\rfloor$ for all $t$ makes it possible to build up momentum in the online setting.

\subsection{Universal lower bound}
In Theorem \ref{thm:firstorder} we give a generalization of Nesterov's lower bound for the online setting considered in this paper. In the proof, we omit certain details that can be found in \cite[Sec. 2.1.4]{Nesterov2}.

\begin{theorem} \label{thm:firstorder}
Let $\Hi=\ell_2(\N)$. For any $x_0$, $\exists (f_t)\in \SS(\kappa^{-1},L,\sigma)$ such that, if $(x_t)$ is generated by an online first-order method starting at $x_0$, then
\begin{align*}
    \|x_t-x_t\opt\| \geq \frac{\sqrt{\kappa}-1}{2}\sigma.
\end{align*}
\end{theorem}
\begin{proof}
First, set $\gamma=\frac{\sqrt{\kappa}-1}{\sqrt{\kappa}+1}$ and let $c$ be the solution to $\gamma^c=\sigma\sqrt{\frac{1+\gamma}{1-\gamma}}$, namely $c=\log\left(\sigma \sqrt{\frac{1+\gamma}{1-\gamma}}\right)/\log(\gamma)$. Note that $\frac{\gamma}{\sqrt{1-\gamma^2}}=\frac{\sqrt{\kappa}-1}{2}\kappa^{-1/4}=\frac{\sqrt{\kappa}-1}{2}\sqrt{\frac{1-\gamma}{1+\gamma}}$. Set $\mu=\kappa^{-1}L$, define $A$ as the symmetric tridiagonal operator on 
$\ell_2(\N)$
with $2$'s on the diagonal and $-1$'s on the sub-diagonal, and let $a\in [\mu,L]$. Abusing notation to write the operators as matrices, define
\begin{align*}
f_t(x) &= \frac{1}{2}x^T\begin{pmatrix}
aI_t & \vline & 0\\
\hline
0 & \vline & \frac{L-\mu}{4}A+\mu I
\end{pmatrix}x-\gamma^c x^T\begin{pmatrix}
a1_t\\
\hline
\frac{L-\mu}{4}e_1
\end{pmatrix}.\\
\text{Then, }\nabla f_t(x) &= \begin{pmatrix}
aI_t & \vline & 0\\
\hline
0 & \vline & \frac{L-\mu}{4}A+\mu I
\end{pmatrix}x-\gamma^c \begin{pmatrix}
a1_t\\
\hline
\frac{L-\mu}{4}e_1
\end{pmatrix}\\
\text{so }x_t\opt(i)&=\begin{cases} \gamma^c & i\leq t\\ \gamma^{i-t+c} & i>t\end{cases}\\
\text{and so }\|x_{t+1}\opt-x_t\opt\|^2 &= \gamma^{2c}\frac{1-\gamma}{1+\gamma}.
\end{align*}
Thus, $(f_t)\in \SS(\kappa^{-1},L,\sigma)$. Without loss of generality, assume that $x_0=0$ since shifting $(f_t$) does not affect membership in $\SS(\kappa^{-1},L,\sigma)$. Then, $x_t(i)=0~\forall i>t$ for any  first-order online algorithm. Thus,
\begin{align*}
\|x_t-x_t\opt\| &\geq \left(\sum_{i=t+1}^{\infty} \gamma^{2(i-t+c)}\right)^{1/2}\\
&= \gamma^{c} \frac{\gamma}{\sqrt{1-\gamma^2}}\\
&= \frac{\sqrt{\kappa}-1}{2}\gamma^c\sqrt{\frac{1-\gamma}{1+\gamma}}\\
&= \frac{\sqrt{\kappa}-1}{2}\sigma.
\end{align*}
\end{proof}

\begin{remark}
\label{remark}
If $\kappa\geq 5$, then $\frac{L-\mu}{4}\in[\mu,L]$. If we apply $\ALG(4/(L-\mu),0,0)$ to the online Nesterov function with $a=\frac{L-\mu}{4}$, then it is easy to see that $\|x_t-x_t\opt\|=\frac{\sqrt{\kappa}-1}{2}\sigma$. Thus, online gradient descent with an appropriately large step-size performs optimally against the online Nesterov function.
\end{remark}

In the following subsection, we will construct an algorithm that performs optimally up to a fixed constant against the full class $\SS(\kappa^{-1},L,\sigma)$; that is, it exhibits an upper bound that is equal to the lower bound of Theorem \ref{thm:firstorder} times a fixed constant.

\subsection{Online long-step Nesterov's method} \label{subsec:OLNM}

There is a conceptual difficulty when it comes to adapting accelerated methods to the online setting. Informally, in batch optimization, ``acceleration'' refers to the fact that accelerated methods converge faster than gradient descent. However, the goal in the online optimization framework considered here is reduced tracking error, and not necessarily faster convergence. As shown by the rotating quadratic example, tracking and convergence actually behave differently. Fortunately, we can leverage the fast convergence of Nesterov's method towards reduced tracking error. 

In this section, we present a long-step Nesterov's method; the term ``long-steps'' refers to the fact that the algorithm takes a certain number of steps using the same stale function before catching up to the most recent function and repeating. For this particular long-step Nesterov's method, we are able to prove upper bounds on the tracking error (on the other hand, no bounds for the online Nesterov's method are yet available, and are the subject of current efforts). 

The specific sequence constructed by $\OLNM(T)$ is defined in Algorithm~\ref{algo:OLNM}.
\begin{algorithm}[h]
\begin{algorithmic}[1]
\Require $x_0$
\State $y_0\gets x_0$, $z_0\gets x_0$, $a_0\gets 1$
\For{$t=1,2,\ldots$} 
          \State $z_{t+1} =y_t-\frac{1}{L}\nabla f_{T\lfloor t/T \rfloor}(y_t)$
 \If{     $T \not| ~ t+1$ }
 \State $a_{t+1} = \frac{1+\sqrt{1+4a_t^2}}{2}$
 \State $y_{t+1} = z_{t+1} +\frac{a_t-1}{a_{t+1}}(z_{t+1}-z_t)$
 \State $x_{t+1} = x_t$ \Comment{so $x_{t+1} = x_t = x_{T\lfloor t/T\rfloor}$}
 \ElsIf{$T\mid t+1$}
 \State $a_{t+1} = 1$
 \State $y_{t+1} = z_{t+1}$
 \State $x_{t+1} = z_{t+1}$
 \EndIf
\EndFor
\end{algorithmic}
\caption{$\OLNM(T)$}
\label{algo:OLNM}
\end{algorithm}

The method can be extended to the composite case by applying the proximal operator to the $z$ iterates. Theorem \ref{thm:onlineNesterov} gives an upper bound on the tracking iterate error of $\OLNM$, using results from \cite[Thm. 10.34]{Beck}.

\begin{theorem} \label{thm:onlineNesterov}
If $(f_t)\in \SS(\kappa^{-1},L,\sigma)$, then, given $x_0$, $\OLNM(T)$ where $T=c\sqrt{\kappa}$ for $c>2$ such that $c\sqrt{\kappa}\in\N$, constructs a sequence $(x_t)$ such that
\begin{align} \label{eq:OLNM}
    \limsup_{t\to\infty}\|x_t-x_t\opt\|\leq \frac{2c(c-1)}{c-2}\sqrt{\kappa}\sigma.
\end{align}
\end{theorem}
\begin{proof}
First, via standard batch optimization bounds, we have
\begin{align*}
    \|x_{kT}-x_{(k-1)T}\opt\|\leq \frac{2\sqrt{\kappa}}{T}\|x_{(k-1)T}-x_{(k-1)T}\opt\|.
\end{align*}
Thus,
\begin{align*}
    \|x_{kT}-x_{kT}\opt\|&\leq\frac{2\sqrt{\kappa}}{T}\|x_{(k-1)T}-x_{(k-1)T}\opt\|+T\sigma\\
    &\leq \left(\frac{2\sqrt{\kappa}}{T}\right)^k\|x_0-x_0\opt\|+\frac{T\sigma}{1-\frac{2\sqrt{\kappa}}{T}}\\
    &= \left(\frac{2\sqrt{\kappa}}{T}\right)^k\|x_0-x_0\opt\|+\frac{T^2\sigma}{T-2\sqrt{\kappa}}
\end{align*}
and so
\begin{align*}
    \|x_t-x_t\opt\| &= \|x_{T\lfloor t/T \rfloor}-x_t\opt\|\\
    &\leq \|x_{T\lfloor t/T \rfloor}-x_{T\lfloor t/T \rfloor}\opt\|+\|x_t\opt-x_{T\lfloor t/T \rfloor}\opt\|\\
    &\leq \left(\frac{2\sqrt{\kappa}}{T}\right)^{\lfloor t/T \rfloor}\|x_0-x_0\opt\|+\frac{T^2\sigma}{T-2\sqrt{\kappa}}+T\sigma.
\end{align*}
Then, taking the limit supremum, we get
\begin{align*}
    \limsup_{t\to\infty}\|x_t-x_t\opt\| &\leq \frac{T^2\sigma}{T-2\sqrt{\kappa}}+T\sigma\\
    &= \frac{2T\sigma\left(T-\sqrt{\kappa}\right)}{T-2\sqrt{\kappa}}\\
    &= \frac{2c(c-1)}{c-2}\sqrt{\kappa}\sigma.
\end{align*}
\end{proof}

If we minimize the bound in Eq.~\eqref{eq:OLNM} over $c\in\R$, then we get $c=2+\sqrt{2}$ with a value of $2c(c-1)/(c-2)=6+4\sqrt{2}\approx 11.66$; this is in contrast with the batch setting, where $c=\sqrt{8}$. However, we have the extra restriction that $c\sqrt{\kappa}\in\N$ so, in general, we will take $T=\lfloor (2+\sqrt{2})\sqrt{\kappa}\rfloor$.

Note that the bound is asymptotically (as the condition number goes to infinity) optimal (hence tight) up to the constant $4c(c-1)/(c-2)$. In particular, for $\kappa \geq (4c(c-1)/(c-2))^2$, the bound is better than the bound for online gradient descent. However, these are bounds over a general class of functions. One question is how the two methods fare against specific examples, as exemplified next.

As we noted in Remark \ref{remark}, online gradient descent with an appropriately large step-size performs optimally against the online Nesterov function. Even with a typical step-size, the tracking iterate error of online gradient descent still scales linearly with the square root of the condition number. Furthermore, while $\OLNM$ also scales linearly with the square root of the condition number, online gradient descent has a smaller constant. Figure \ref{fig:adversary}(a) depicts the iterate error for algorithms applied to the online Nesterov function with condition number equal to 500. Online gradient descent performs the best, followed by online Nesterov's method and $\OLNM$. In fact, this can be seen in the dependence on the condition number as well. Figure \ref{fig:adversary}(b) shows the linear dependence of the tracking iterate error on the square root of the condition number. The constant for online gradient descent is 0.481, the constant for online Nesterov's method is 1.101, and the constant for $\OLNM$ is 2.491. Note that, despite $\OLNM$ having a worse constant than online gradient descent, the former's constant is still less than its upper bound of $2c(c-1)/(c-2)\approx 11.66$.

\begin{figure}[H]
\subfigure[]{\includegraphics[width=.48\linewidth]{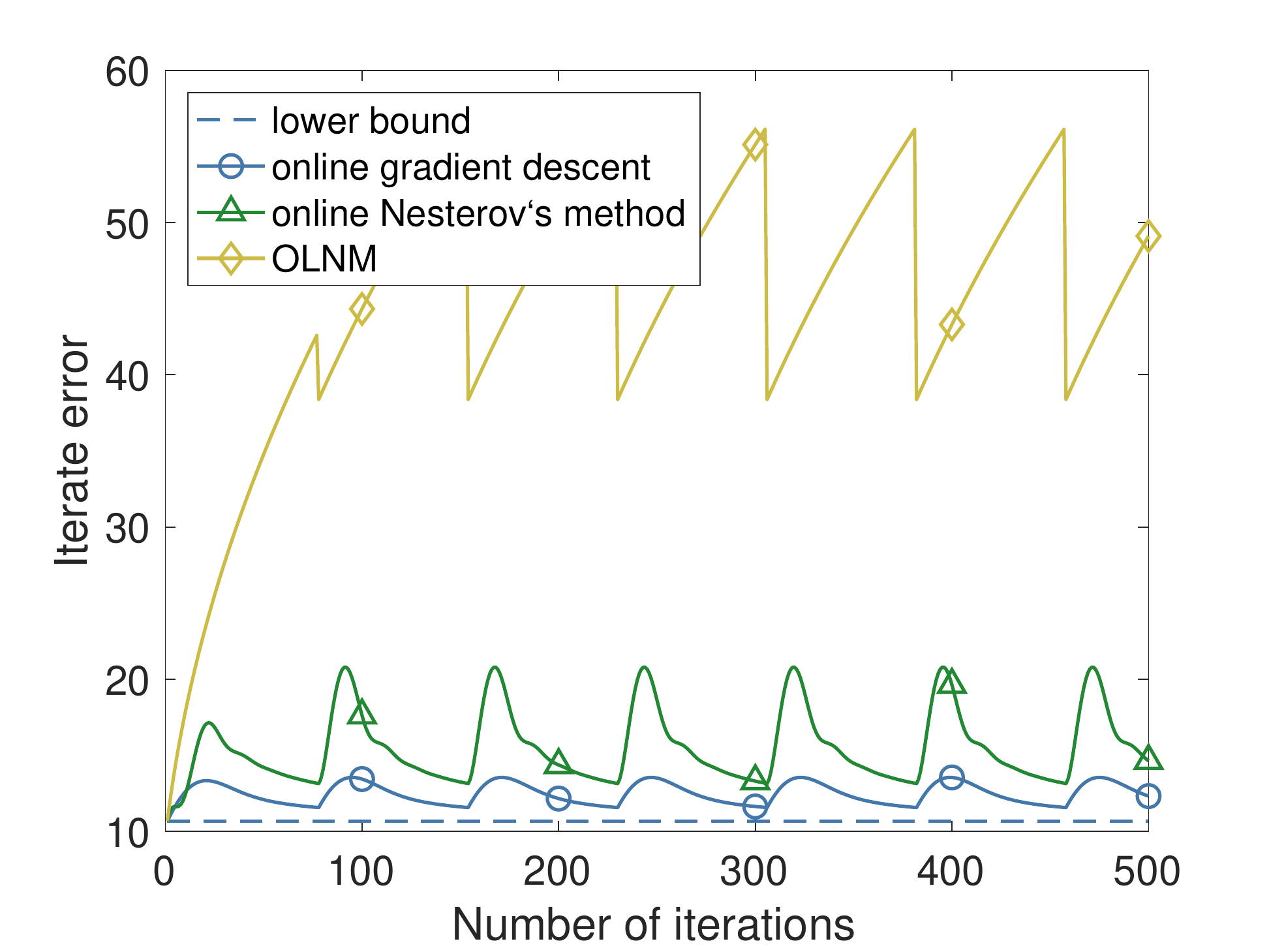}}
\subfigure[]{\includegraphics[width=.48\linewidth]{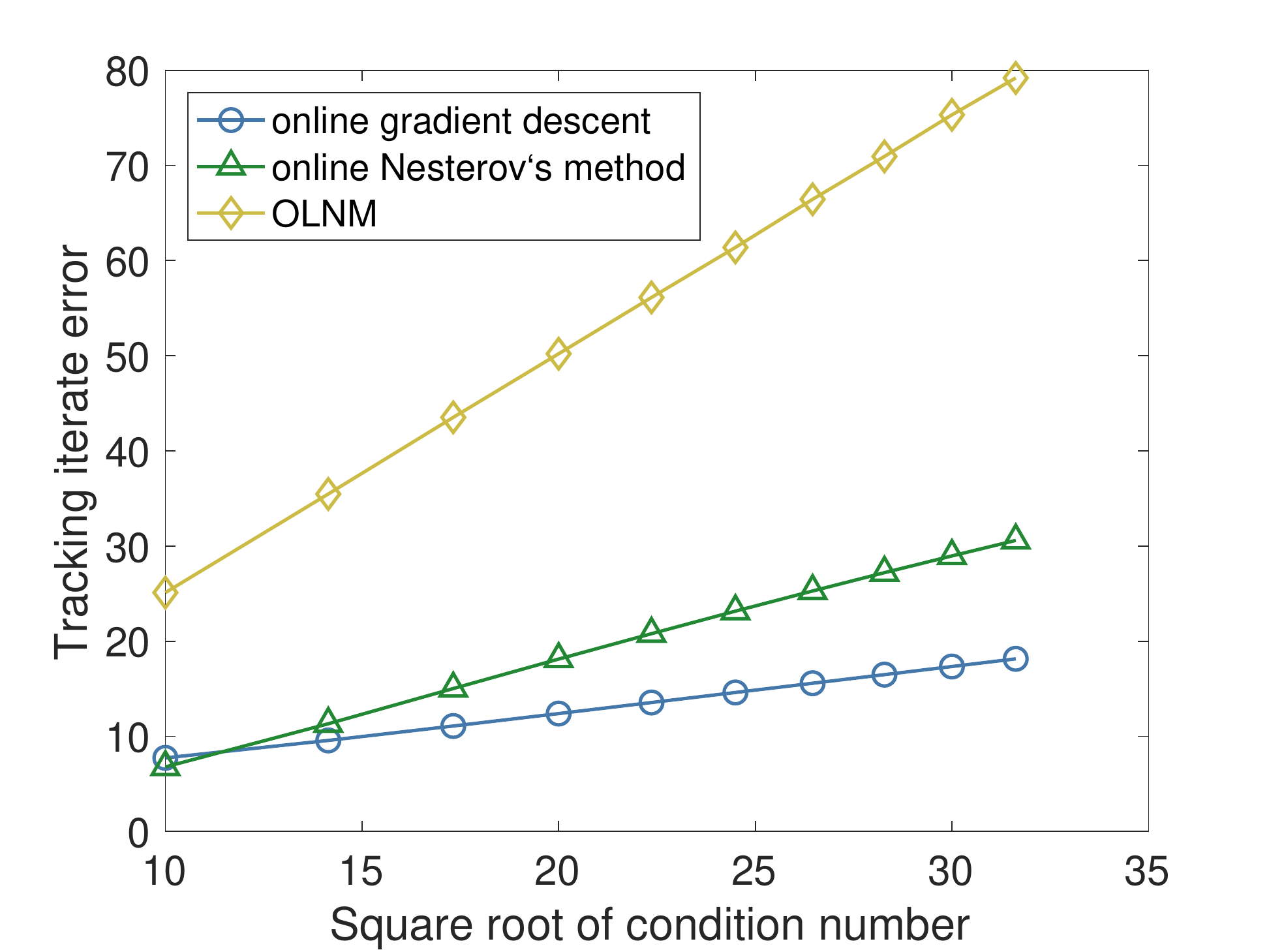}}
\caption{Algorithms applied to the online Nesterov function with $L=500$, $d=1000$, $a=(L+\mu)/2$, $\sigma =1$, and $T=\lfloor(2+\sqrt{2})\sqrt{\kappa}\rfloor$. (a) Evolution of the iterate error for the particular example $\mu=1$. (b) Tracking iterate error for varying $\mu$.}
\label{fig:adversary}
\end{figure}

For the rotating quadratic example, the minimizer is fixed, so $\OLNM$ has the same convergence rate as Nesterov's method does in the batch setting. However, with the right step-size, online gradient descent can converge in just two steps! On the other hand, for the translating quadratic example, since the tracking iterate error of online gradient descent scales with $\kappa$, while the tracking iterate error of $\OLNM$ scales with $\sqrt{\kappa}$, $\OLNM$ outperforms online gradient descent for sufficiently high condition number. In particular, Figure \ref{fig:translating}(a) depicts the iterate error for algorithms applied to the translating quadratic function with condition number equal to 500. Online Nesterov's method performs the best, followed by $\OLNM$ and online gradient descent. Figure \ref{fig:translating}(b) shows the tracking iterate error for varying condition number for $\OLNM$ (online gradient descent and online Nesterov's method are left out since we analytically solved for their tracking iterate error). The constant is 7.21. Note that this is larger than the constant for $\OLNM$ applied to the online Nesterov function. While the online Nesterov function is a universal adversary, the translating quadratic is more particularly adversarial for $\OLNM$ because the minimizer is maximizing its distance away from the $\OLNM$ iterates by moving in a straight line away from the old minimizer the $\OLNM$ iterates are approaching. Also note that this is more than half of the upper bound, showing that the upper bound is tight at least up to a constant less than two.

\begin{figure}[H]
\subfigure[]{\includegraphics[width=.48\linewidth]{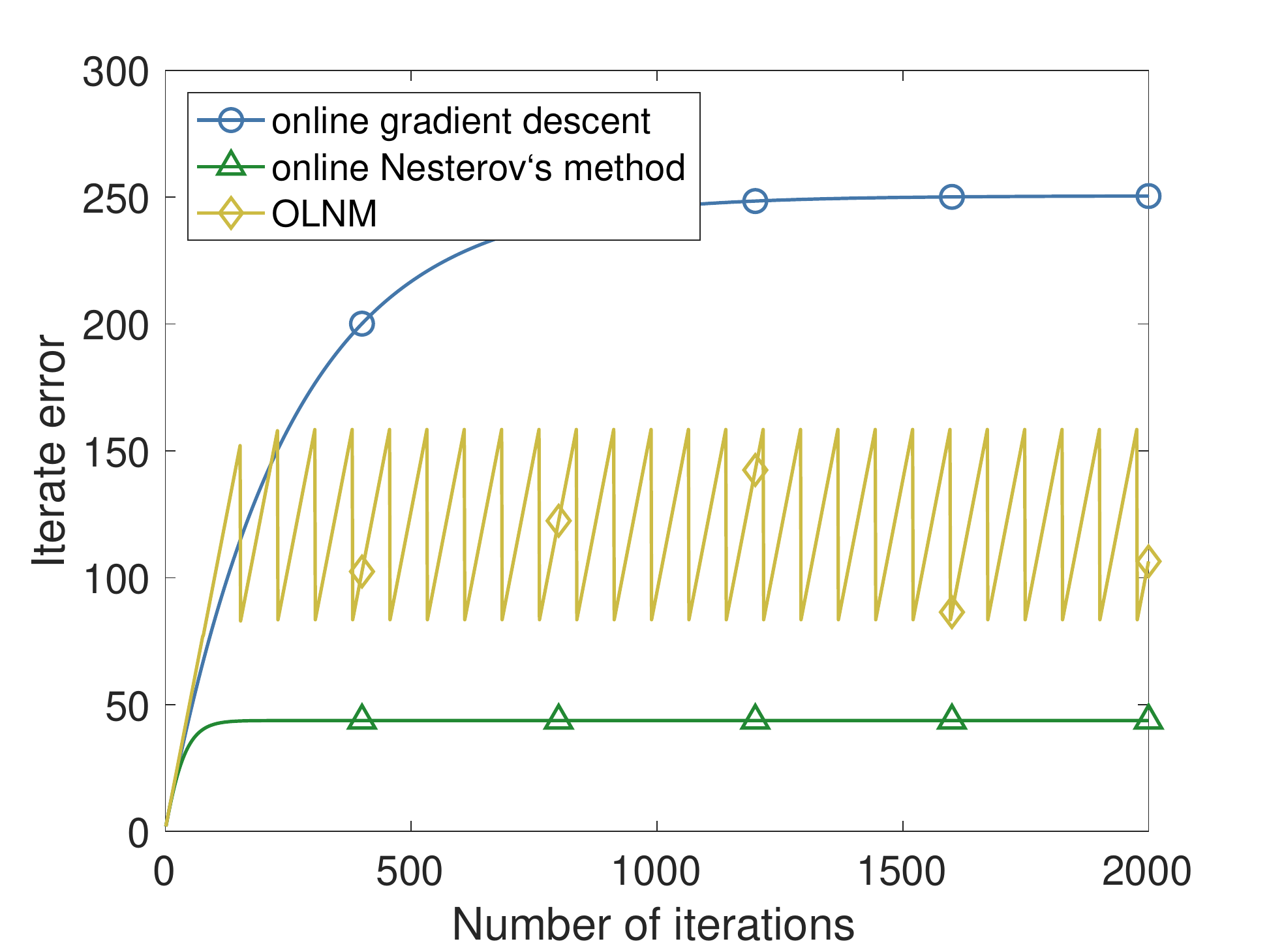}}
\subfigure[]{\includegraphics[width=.48\linewidth]{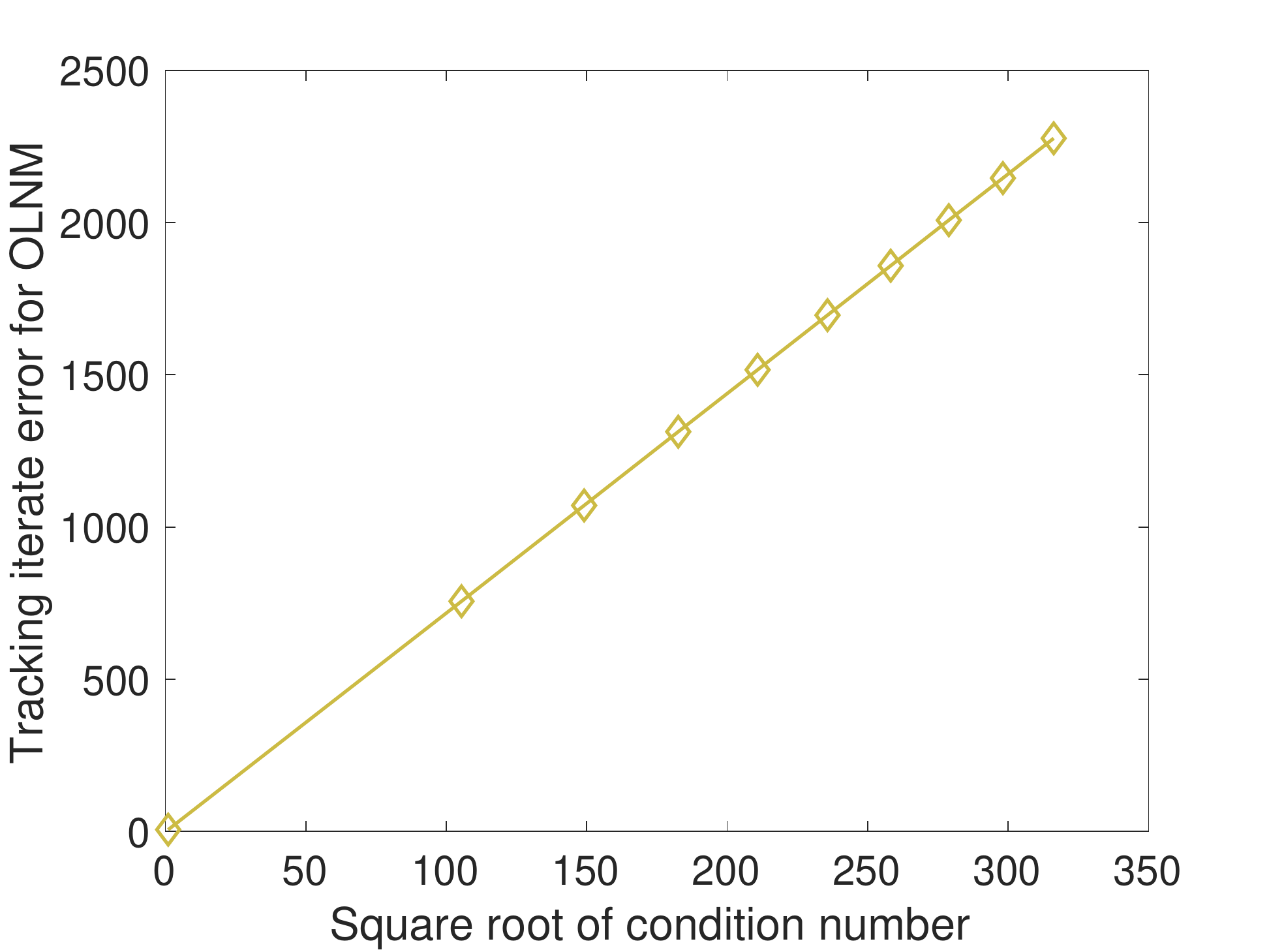}}
\caption{Algorithms applied to the translating quadratic function with $d=2$, $\sigma =1$, and $T=\lfloor(2+\sqrt{2})\sqrt{\kappa}\rfloor$. (a) shows the evolution of the iterate error for $L=500$ and $\mu=1$. (b) shows the tracking iterate error for $L=1$ and varying $\mu$.}
\label{fig:translating}
\end{figure}

\section{Regularization} \label{sec:regularization}

When the functions are convex, but not strongly convex, performance metrics for online first-order methods typically rely on dynamic regret bounds. The dynamic regret is the \textit{averaged} function error. We take a different approach, however, deriving a gradient error bound via regularization. While gradient error bounds are weaker than function error bounds, the benefit is that we bound the error, rather than the averaged error.

The main idea is that there is a trade-off between tracking error and regularization error. Regularizing by any amount means that we can apply Theorem \ref{thm:stronglyConvex1}. In this case, increasing the amount of regularization increases the regularization error while also decreasing the tracking error. In the following, we provide a framework for balancing these two errors.

Given $x_0$, let \emph{online regularized gradient descent,} $\ORGD(\delta,x_c)$, with $\delta>0$, construct the sequence $(x_t)$ via
\begin{align*} \label{eq:ORGD}
    x_{t+1} = x_t-\frac{2}{L+2\delta}(\nabla f_t(x_t)+\delta(x_t-x_c)). \tag{$\ORGD(\delta,x_c)$}
\end{align*}
It is easy to see that $\ORGD(\delta,x_c)$ is vanilla online gradient descent for the regularized problem $f_t(\cdot;\delta,x_c)=f_t+\frac{\delta}{2}\|\cdot-x_c\|^2$. Since we don't vary $x_c$ in the analysis, we write $f_t(\cdot;\delta)$ for simplicity.

Now, in order to bound the algorithm error of $\ORGD$, we need to bound the regularization error. As with the algorithm error, we can measure the regularization error in terms of the variable, the cost, or the gradient. Unfortunately, it is impossible, without further assumptions, to bound the variable regularization error, $\|x\opt-x_r\opt\|$ where $x_r\opt$ is the unique minimizer to the regularized problem \cite{Rockafeller,Nedic}. For example, if we regularize a constant function by any amount, then there are minimizers arbitrarily far away from the regularized minimizer. However, if we assume the function is coercive ($\|x\|\to\infty \implies f(x)\to\infty$), then we know the variable regularization error \textit{is} bounded. But, it is still impossible to bound it in terms of the strong smoothness constant. For example, assume that we are going to regularize by adding $\frac{\delta}{2}\|\cdot\|^2$. Then, for an arbitrary distance $\DD$, there exists an L-smooth function such that the variable regularization error is greater than $\DD$. In particular, consider $f(x)=\frac{\lambda}{2}\|x-2\DD e\|^2$ where $\lambda\leq \min\{L,\delta\}$ and $e$ is a unit vector. Then $x_r\opt=\frac{\lambda}{\lambda+\delta}e$ and so $\|x\opt-x_r\opt\|=\frac{\delta}{\lambda+\delta}\|2\DD e\|\geq \DD$.

Fortunately, even without coercivity, it \textit{is} possible to bound the gradient regularization error \cite[Thm. 2.2.7]{Nesterov2}, which allows us to bound the tracking gradient error of $\ORGD$. However, we do have to make an additional assumption. Loosely, we have to assume that the sequence of minimizer sets doesn't ``drift.'' To make the assumption more precise, we need some definitions.

First, let $x_t\opt(\delta)=\argmin_x\, f_t(x;\delta)$ for $\delta>0$. We know the right-hand side is a singleton by strong convexity. Furthermore, since $f_t$ is proper, closed, and convex, $\partial f_t$ is maximally monotone, and so we can apply \cite[Thm. 23.44]{Combettes}, which tells us that $x_t\opt(\cdot)$ is continuous and $\lim_{\delta\to 0}x_t\opt(\delta)$ is the unique projection of $x_c$ onto the zero set of $\partial f_t$. Thus, we can define $x_t\opt(0)=\lim_{\delta\to 0}x_t\opt(\delta)$. We also have that $\delta \mapsto \|x_t\opt(\delta)-x_c\|$ is monotonically decreasing (this is not hard to show and can be found in the proof of \cite[Thm. 2.2.7]{Nesterov2}). Let $R(\delta;x_c)=\sup_{t\in\N\cup\{0\}}\|x_t\opt(\delta)-x_c\|$ for all $\delta\geq 0$. Again, since we do not vary $x_c$ in the analysis, we write $R(\delta)$ for simplicity. Note that $R(\cdot)$ is also monotonically decreasing. Thus, if $R(0)<\infty$, then $R(\delta)<\infty$ for all $\delta\geq 0$. We will assume this is true:
\begin{align*} \label{eq:assump}
    R(0)<\infty.
    \tag{bounded drift}
\end{align*}
While this assumption precludes problems like the translating quadratic, it is realistic when the problem is data-dependent. For machine learning problems it is common to have normalized data and to be learning normalized weights. Then, the minimizer will, in fact, lie in a bounded set.

Let $\sigma(\delta)=\sup_{t\in\N}\|x_t\opt(\delta)-x_{t-1}\opt(\delta)\|$ for $\delta\geq 0$. Note that $\sigma(\delta)$ is bounded above by $2R(\delta)$, via the triangle inequality, which in turn is bounded above by $R(0)<\infty$, via the monotonicity of $R(\cdot)$.

Finally, consider the function $h(\delta)=\frac{\sigma(\delta)}{R(\delta)}-2\left(\frac{\delta}{L}\right)^2$ for $\delta\geq 0$. $h(\cdot)$ is continuous since $x_t\opt(\cdot)$ is continuous and $R(\delta)>0$ for all $\delta\geq 0$ (unless $x_c=x_t\opt(0)$, which would be the trivial case). Also, $h(0)=\frac{\sigma(0)}{R(0)}>0$ and $h(L)\leq 0$. Thus, via the Intermediate Value Theorem, we have just proved the following lemma.

\begin{lemma}
\label{thm:lemma}
If $(f_t)\in \SS'(L)$ has bounded drift, then $\exists~ 0<\delta\leq L$ s.t. $\delta=L\sqrt{\frac{\sigma(\delta)}{2R(\delta)}}$.
\end{lemma}

In Theorem \ref{thm:onlineRegularization}, we derive a bound on the tracking gradient error in terms of $\sigma(\delta)$ and $R(\delta)$. However, since $\sigma(\delta)$ and $R(\delta)$ both depend on $\delta$, it is impossible, without further information about the function sequence, to minimize the bound explicitly with respect to $\delta$. The $\delta$ in Lemma \ref{thm:lemma} corresponds to what the minimizing $\delta$ would be if $\sigma(\delta)$ and $R(\delta)$ were constant.

\begin{theorem} \label{thm:onlineRegularization}
If $(f_t)\in \SS'(L)$ has bounded drift, then $\exists ~0<\delta\leq L$ such that, given $x_0$, $\ORGD(\delta,x_c)$ constructs a sequence $(x_t)$ with
\begin{align}
    \label{eq:reg}
    \limsup_{t\to\infty}\|\nabla f_t(x_t)\|\leq 2\sqrt{2}L\sqrt{\sigma(\delta) R(\delta)}.
\end{align}
\end{theorem}
\begin{proof}
Let $\delta$ be as in Lemma \ref{thm:lemma}. Observe,
\begin{align*}
    0&=\nabla f_t(x_t\opt(\delta);\delta)\\
    &=\nabla f_t(x_t\opt(\delta))+\delta(x_t\opt(\delta)-x_c)
\end{align*}
so
\begin{align*}
    \|\nabla f_t(x_t\opt(\delta))\| = \delta \|x_t\opt(\delta)-x_c\|.
\end{align*}
Thus,
\begin{align*}
    \|\nabla f_t(x_t)\|&\leq \|\nabla f_t(x_t\opt(\delta))\|+\|\nabla f_t(x_t)-\nabla f_t(x_t\opt(\delta))\|\\
    &\leq \delta \|x_t\opt(\delta)-x_c\|+L\|x_t-x_t\opt(\delta)\|\\
    &\leq \delta R(\delta)+L\|x_t-x_t\opt(\delta)\|
\end{align*}
so
\begin{align*}
    \limsup_{t\to\infty}\|\nabla f_t(x_t)\|
    &\leq \delta R(\delta) + \frac{L(L+2\delta)\sigma(\delta)}{2\delta}\\
    &= L\sqrt{2\sigma(\delta)R(\delta)}+L\sigma\\
    &\leq L\sqrt{2\sigma(\delta)R(\delta)}+L\sqrt{2\sigma(\delta)R(\delta)}\\
    &= 2\sqrt{2}L\sqrt{\sigma(\delta) R(\delta)}.
\end{align*}
\end{proof}

Note that if $\sigma(0)\approx\sigma(\delta)\approx R(\delta)\approx R(0)$ then the bound in Eq. \ref{eq:reg} is $\approx 2\sqrt{2}LR(0)$, which is of the same form but has worse constants than the bound for the algorithm which abstains from tracking (i.e., 
$\forall t ~x_t=x_c$):
\begin{align*}
    \|\nabla f_t(x_t)\| &= \|\nabla f_t(x_c)\|\\
    &= \|\nabla f_t(x_c)-\nabla f_t(x_t\opt(0))\|\\
    &\leq L \|x_c-x_t\opt(0)\|\\
    &\leq LR(0).
\end{align*}
Conversely, if $\sigma(\delta)\ll R(\delta)$, then $\delta$ is small so $\sigma(0)\approx\sigma(\delta)$ and $R(0)\approx R(\delta)$. In this case, Eq. \ref{eq:reg} becomes $\approx 2\sqrt{2}\sqrt{\frac{\sigma(0)}{R(0)}}LR(0)\ll LR(0)$. Thus, we can only guarantee the usefulness of $\ORGD$ when $\sigma(\delta)\ll R(\delta)$.

\section{Illustrative numerical examples} \label{sec:numerics}

Our code is online at \href{https://github.com/liammadden/time-varying-experiments}{github.com/liammadden/time-varying-experiments}. In this section, we consider time-varying least-squares regression and time-varying logistic regression. At each $t\in\N\cup\{0\}$, we are given an input matrix $A^{(t)}\in\R^{n\times d}$, with $i$-th row vector denoted by $a^{(t)}_i$, and an $n$-dimensional output vector $b^{(t)}$, where each $(a^{(t)}_i,b^{(t)}_i)$ corresponds to a single data point. For least-squares regression, $b^{(t)}\in\R^n$. For logistic regression, $b^{(t)}\in\{-1,1\}^n$. For simplicity, we assume the inputs are constant across time, i.e. $A_t=A$ for all $t$, while only the outputs change. This type of time-variation fits applications with time-series data where the time-dependency is not captured by any of the input features. For example, a problem may have a discrete number of states that it switches between based on a hidden variable.

The cost function for least-squares regression is
\begin{align*}
    f_t(x) &= \frac{1}{2}\sum_{i=1}^n \left(\langle a_i,x\rangle-b_i^{(t)}\right)^2\\
    &= \frac{1}{2}\|Ax-b^{(t)}\|^2
\end{align*}
and for logistic regression is
\begin{align*}
    f_t(x) = \frac{1}{n}\sum_{i=1}^n \log\left(1+\exp(-b^{(t)}_i\langle a_i,x\rangle)\right).
\end{align*}
The least-squares cost function is strongly convex while the logistic cost function is only strictly convex. However, it can be shown that gradient descent still achieves linear convergence when applied to the logistic cost function \cite{plinequality}.

A ``weight vector,'' $x\in\R^d$, predicts an output, $b$, from an input, $a$. For least-squares regression, $b=\langle a,x\rangle$. For logistic regression, $b=\text{sign}(\langle a,x\rangle)$; more precisely, $x$ predicts $b=1$ with probability $\left(1+\exp(-\langle a,x\rangle)\right)^{-1}$ and $b=-1$ with the remaining probability. The minimizer, $x_t\opt$, of the respective objective function is the ``optimal'' weight vector.

\subsection{Least squares regression}

For the least-squares regression we considered the case $n=20$ and $d=5$. We generated $A$ by defining its singular value decomposition. For its left and right-singular vectors, we sampled two Haar distributed orthogonal matrices, $U\in\R^{n\times n}$ and $V\in\R^{d\times d}$. We let its singular values be equally spaced from $1/\sqrt{\kappa}$ to 1. We generated $b_t$ by varying $x_t\opt$ via a random walk with $\sigma=1$ and adding a low-level of Gaussian noise (mean, 0; standard deviation, $10^{-3}$) to the corresponding outputs. We initialized $x_0\opt$ as the vector of ones. For each $\kappa$, we ran the experiment 200 times and took the average of the tracking errors.

For $\OLNM$, we set $T=\lfloor (2+\sqrt{2})\sqrt{\kappa}\rfloor$. However, instead of updating the $x$ iterate by the $z$ iterate only every $T$ iterations, we updated it every iteration. Doing so works better in practice, though we lose the theoretical guarantees of Theorem \ref{thm:onlineNesterov}. However, the theory only requires that the $x$ iterates do not move away from the minimizer that the $z$ iterates are converging to. Not updating the $x$ iterates ensures they do not move away, but it is too conservative in practice.

For all the algorithms, we initialized $x_0=x_0\opt$ so that we wouldn't need a long ``burn-in'' period while waiting for convergence to the tracking error. We computed the tracking error as the maximum error over the 5th cycle of $T$. Figure \ref{fig:randomwalks} shows the results. $\OLNM$ performs better than online gradient descent, however, online Nesterov's method performs much better than both of the other methods, despite its lack of guarantees.

\begin{figure}[H]
    \centering
    \includegraphics[width=.5\linewidth]{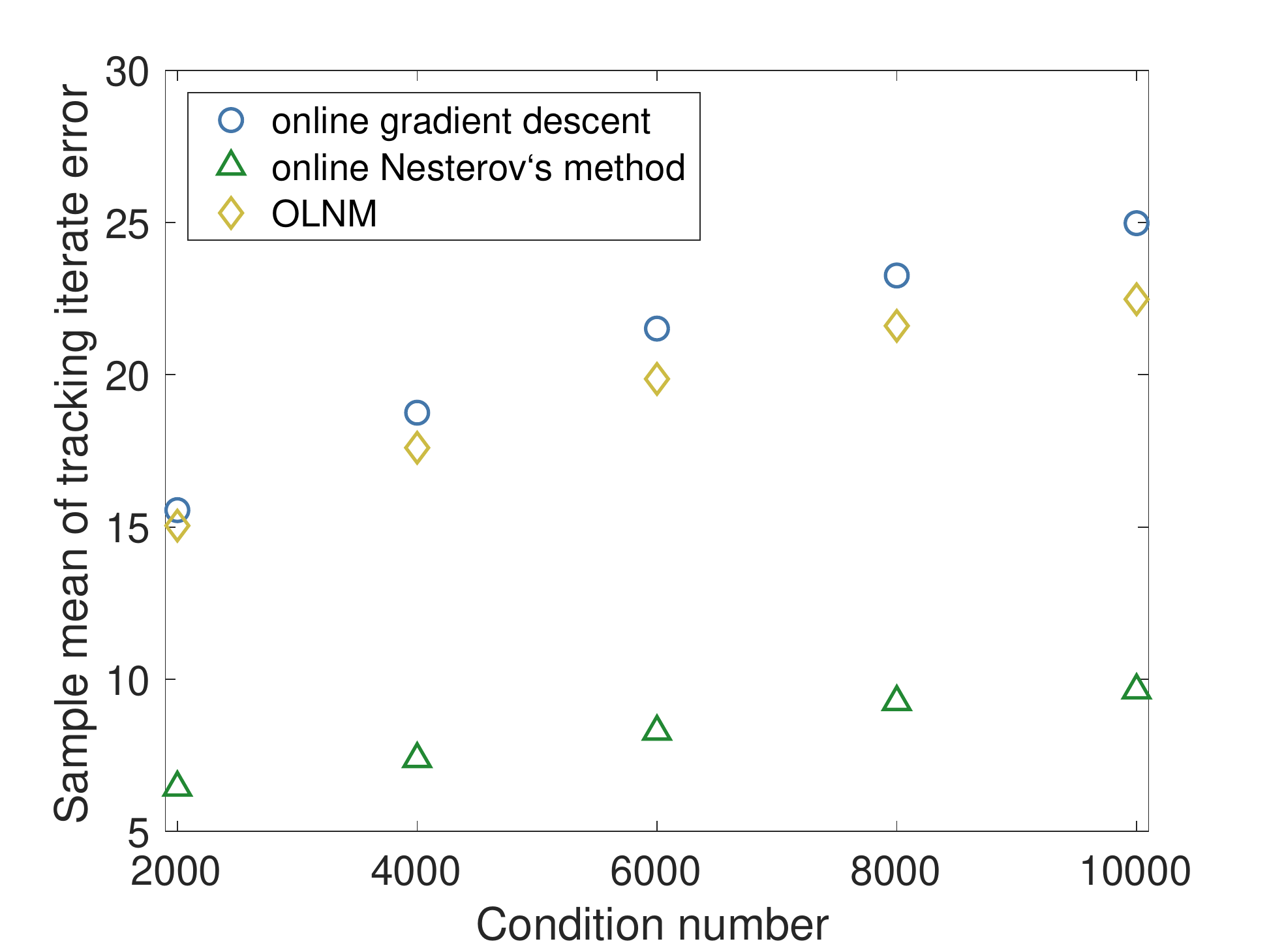}
    \caption{Least-squares regression with random data matrix and random walk variation of the minimizer.}
    \label{fig:randomwalks}
\end{figure}

\subsection{Logistic regression with streaming data}

For logistic regression we again considered the case $n=20$ and $d=5$. We generated $A$ in the same way as for least-squares regression, but with singular values drawn from the uniform distribution on (0,1) and then scaled to have maximum singular value equal to $2\sqrt{L}$ (since the strong smoothness constant of the cost function is $\|A\|_2^2/4$). We initialized $b_t$ from the Rademacher distribution and uniformly at random chose one label to flip each time step.

A classic problem that logistic regression is used for is spam classification. The problem is to predict whether an email is spam or not based on a list of features. Since emails are received in a streaming fashion, this is a time-varying problem. Thus, when an email first arrives, we may report that it is spam, but then later realize it is not, or vice versa. Such an extension of spam classification can be incorporated into the time-varying logistic regression framework.

We let $x_c$ be the zero vector. As a weight vector, $x_c$ predicts $1$ or $-1$ with equal probability. For each $L$, we approximated the solution to the fixed point equation $\delta = L\sqrt{\frac{\E(\sigma(\delta))}{\E(R(\delta))}}$ where the expectation is taken over $A$ and the $b_t$'s. Figure \ref{fig:logistic}(a) shows how $\sigma(\delta)$ and $R(\delta)$ vary with $L$. The ratio between $\sigma(\delta)$ and $R(\delta)$ stays constant at $\approx 0.5$.

We ran the experiment 200 times and took the average of the tracking errors. We calculated the tracking error using windows of 100 iterations and stopping when the maximum error over the last two windows was the same as over the last window. Figure \ref{fig:logistic}(b) shows the results. Online gradient descent (which lacks theoretical tracking gradient error bounds) performs slightly better than $\ORGD$ and both outperform $x_c$. Thus, with only one label flipping every time step, it is still possible to track the solution. However, if we sufficiently increased the number of labels that flip every time step, then it would be better to just guess the labels, as $x_c$ does.

\begin{figure}[H]
\subfigure[]{\includegraphics[width=.5\linewidth]{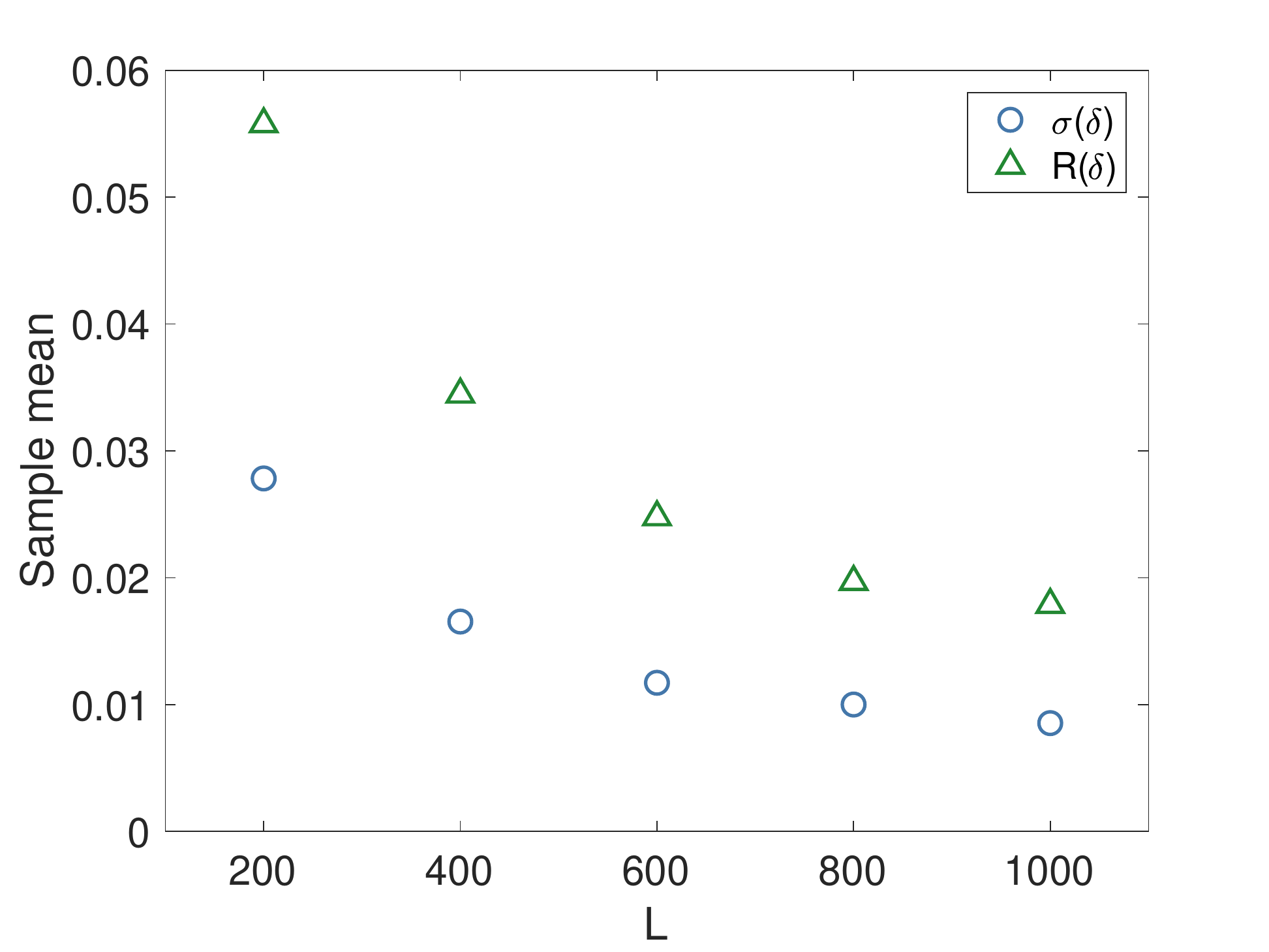}}
\subfigure[]{\includegraphics[width=.5\linewidth]{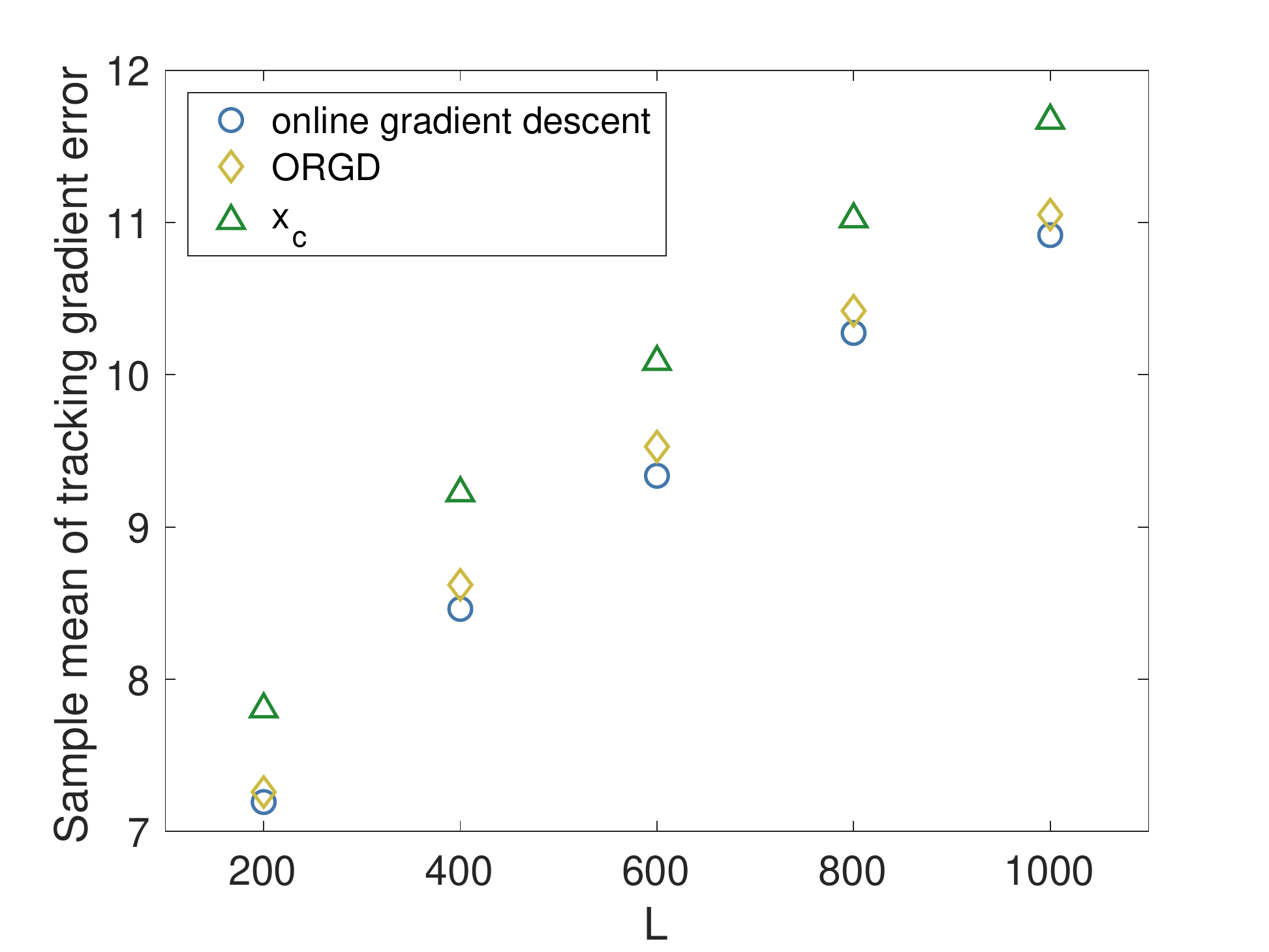}}
\caption{Logistic regression with random data matrix and  randomly flipping labels.}
\label{fig:logistic}
\end{figure}

\rev{
\section{Conclusions} \label{sec:conclusion}
}
By categorizing classes of functions based on the minimizer variation, this paper was able to generalize results from batch optimization to the online setting. These results are important for time-varying optimization, especially for applications in power systems, transportation systems, and communication networks. We showed that fast convergence does not necessarily lead to small tracking error. For example, online Polyak's method can diverge even when the minimizer variation is zero. On the other hand, online gradient descent is guaranteed to have a tracking error that does not grow more than linearly with the minimizer variation. We also gave a universal lower bound for online first-order methods and showed that $\OLNM$ achieves it up to a constant factor. It is a future research direction to consider more deeply the connection to the long-steps of interior-point methods and see if a satisfactory criteria can be found for adaptively deciding when to long-step. Perhaps, this enquiry may eventually lead to error bounds for online Nesterov's method itself.

\section*{Acknowledgements}
All three authors gratefully acknowledge support from the NSF program ``AMPS-Algorithms for Modern Power Systems'' under award \# 1923298.

\bibliographystyle{spmpsci_unsrt}
\bibliography{main}

\end{document}